\documentclass[11pt]{amsart}
\usepackage{amssymb,amsthm}
\usepackage[all]{xy}
\usepackage{comment}

\numberwithin{equation}{section}
\newtheorem{thm}[equation]{Theorem}
\newtheorem{lem}[equation]{Lemma}
\newtheorem{prop}[equation]{Proposition}
\newtheorem{cor}[equation]{Corollary}

\theoremstyle{remark}
\newtheorem{rem}[equation]{Remark}

\theoremstyle{definition}
\newtheorem{defn}[equation]{Definition}

\newenvironment{changemargin}[2]{%
  \begin{list}{}{%
    \setlength{\topsep}{0pt}%
    \setlength{\leftmargin}{#1}%
    \setlength{\rightmargin}{#2}%
    \setlength{\listparindent}{\parindent}%
    \setlength{\itemindent}{\parindent}%
    \setlength{\parsep}{\parskip}%
  }%
  \item[]}{\end{list}}

\DeclareMathOperator{\coh}{H}
\DeclareMathOperator{\Hom}{Hom}

\DeclareMathOperator{\Ext}{Ext}

\DeclareMathOperator{\Irr}{Irr}

\DeclareMathOperator{\Proj}{Proj}

\DeclareMathOperator{\Ind}{ind}

\DeclareMathOperator{\Thick}{Thick}

\DeclareMathOperator{\stmod}{stmod}
\DeclareMathOperator{\Spec}{Spec}

\DeclareMathOperator{\Coind}{coind}

\newcommand{\ul}{\underline{\lambda}}

\newcommand{\Z}{\mathbb{Z}}

\newcommand{\ot}{\otimes}
\newcommand{\ua}{\uparrow}
\newcommand{\da}{\downarrow} 
\newcommand{\cF}{\mathfrak F}
\newcommand{\cA}{\mathcal A}
\newcommand{\cB}{\mathcal B}
\newcommand{\cC}{\mathcal C}
\newcommand{\bP}{\mathbb P}

\newcommand{\DOT}{\setlength{\unitlength}{1pt}\begin{picture}(2.5,2)(1,1)\put(2,2.5){\circle*{2}}\end{picture}} 

\newcommand{\fu}{\mathfrak u}

\newcommand{\fsl} {\mathfrak {sl}}

\title[Tensor ideals and varieties]{Tensor ideals and varieties for 
modules of quantum elementary abelian groups}
\author{Julia Pevtsova}
\address{Department of Mathematics, University of Washington, Seattle, WA
98195, USA}
\email{julia@math.washington.edu}
\thanks{This material is based upon work supported by the National Science
Foundation under grant No.\ 0932078000, while the second author was in 
residence at the Mathematical Sciences Research Institute (MSRI) in Berkeley,
California, during the semester of Spring 2013. The first author was
supported by NSF grant DMS-0953011, and 
the second author
by NSF grant DMS-1101399.}
\author{Sarah Witherspoon}
\address{Department of Mathematics, Texas A\&M University, College Station,
TX 77843, USA}
\email{sjw@math.tamu.edu}

\date{January 26, 2014}

\begin{document}

\begin{abstract} 
In a previous paper we constructed rank and support variety theories 
for   ``quantum elementary abelian groups,"
that is, tensor products of copies of Taft algebras. 
In this paper we use both variety theories to 
classify the thick tensor ideals in the stable module
category,  and to  prove a tensor product property for the support varieties. 
\end{abstract}

\maketitle
\section{Introduction} 

This paper is a sequel to our study of rank and support varieties for ``quantum elementary abelian groups" \cite{PvW09}.  
These are arguably the simplest examples of  finite dimensional non-commutative non-cocommutative Hopf algebras. They come up in many different contexts and have several incarnations. For the purposes of this paper, let $\ell\geq 2$ be an
integer and define A to be a semi-direct (or smash) product 
of  a truncated polynomial algebra with an abelian group,  
\[A := k[X_1, \ldots, X_n]/(X_1^\ell, \ldots, X_n^\ell) \rtimes (\Z/\ell\Z)^{\times n}, \]
where the characteristic of $k$ is either zero or relatively prime to $\ell$
(details in Section~\ref{recollections}). 
The algebra $A$ also has the following alternative descriptions. 
\begin{itemize}  
\item[$\circ$]  $A$ is  a 
Borel subalgebra of the small quantum group $\fu_q(\fsl_2^{\oplus n})$,
where $q$ is a primitive $\ell$th root of unity. 
\item[$\circ$] $A$ is  the bosonization of a Nichols algebra $\mathcal B(V)$ where  $V$ is the braided vector space of diagonal type with 
braiding matrix $(q_{ij})$, where  $ q_{ij}$ is 1 if $i=j$, $q$ if $i>j$, and $q^{-1}$ if $i<j$.
\item[$\circ$]  $A$ is the tensor product of $n$ copies of the Taft algebra
of dimension $\ell ^2$. 
\end{itemize}
In \cite{PvW09} we developed the theory of {\it rank}  \, and {\it support} varieties  for finitely generated $A$-modules.   The support varieties were defined classically via the actions of a cohomology ring on the $\Ext$-algebras for $A$-modules $M$. To define rank variety, we considered ``quantum" cyclic shifted subgroups of $A$.  The main theorem was an adaptation of the 
result of Avrunin and Scott \cite{AS} for elementary abelian $p$-groups   to our quantum setting, stating that there is a homeomorphism between the rank and support varieties. 
In this paper we prove two standard results in the theory of varieties for which this Avrunin-Scott type theorem is essential: the classification of tensor ideals in the stable module category of $A$ and the tensor product property for varieties of finitely generated $A$-modules. 
The proofs of the following two theorems are in Sections
\ref{product} and \ref{ideal}, respectively.

\begin{thm}[Tensor product theorem]
\label{thm:tensor} 
Let $M$, $N$  be finitely generated $A$-modules.  Then  $V_A(M\otimes N) = V_A(M)  \cap V_A(N)$, where $V_A(M)$ denotes the support variety of $M$.
\end{thm}

\begin{thm}[Classification of tensor ideals]
\label{thm:class} 
There is an inclusion-preserving bijection between thick tensor ideals of $\stmod A$ and subsets of $\Proj \coh^*(A,k)$ closed under specialization.  
\end{thm} 

A refined version of Theorem \ref{thm:class} is 
Theorem \ref{se:class}.

Our algebras  present themselves as the very first examples of  non-commutative, non-cocommutative   Hopf algebras for which   the tensor product theorem is known. 
Examples of non-commutative non-cocommutative Hopf algebras for which the classification of thick tensor ideals is known are also scarce, although such a classification  was recently given by Benson and the second author \cite{BW} 
using completely different methods
for some Hopf algebras constructed from finite groups.

Due to difficulties that arise in working with a non-symmetric tensor product,  there remain unanswered questions. 
We still do not know whether infinite dimensional modules satisfy the tensor product property for their varieties, even though we prove it for finite dimensional modules using a combination of homological and  representation theoretic arguments. The  original solution of the classification problem for finite groups (or finite group schemes) involved Rickard idempotent modules, the representation theoretic interpretation of the Bousfield localization technique in topology (see \cite{BCR97} and \cite{FP07}). Without the tensor product theorem for infinite dimensional modules, we cannot apply Rickard idempotents to classify thick tensor ideals of $\stmod A$. Instead, we employ ideas from the recent paper~\cite{CI} of Carlson and Iyengar to circumvent the use of infinite dimensional modules and give a very simple proof of the classification. 

The paper is organized as follows. In Section~\ref{recollections}  we introduce notation and briefly 
recall definitions  and basic properties of rank and cohomological support varieties for quantum elementary abelian groups.  Most of the material in that section comes from \cite{PvW09}. In Section~\ref{se:dual} we collect several elementary facts about $A$-modules which culminate in the equality of $V_A(M)$ and $V_A(M^\#)$, where $M^{\#}$ is an $A$-module dual to $M$. Even though the facts proven in that section might feel very familiar to the reader, one should be cautious when working with non-symmetric monoidal module categories in view of the recent results of Benson and the second author: There are examples in \cite{BW} of various anomalies in the behavior of support varieties for modules in such categories.  In particular, it is shown  that the varieties of $M$ and its dual need not coincide in general.  Motivated by these examples, we take extra care with proofs of  ``elementary" facts in Section~\ref{se:dual}. 

In Section~\ref{product} we prove Theorem~\ref{thm:tensor}, the tensor product property. This proof naturally splits into two parts. We first prove the ``easy"  inclusion $V_A(M\otimes N) \subset V_A(M) \cap V_A(M)$ using exclusively the fact that the varieties of $M$ and $M^\#$ coincide. The second part of this section is occupied by the proof of the other inclusion which uses the structural properties of $A$ in an essential way. Finally, in the last section we classify the thick tensor ideals of $\stmod A$.

Throughout this paper, $k$ will denote an algebraically closed field containing a
primitive $\ell$th root of unity $q$; in particular, $\ell$ is not divisible by the characteristic of
$k$. All tensor products and dimensions will be over $k$,  
and all modules are finitely generated left modules, unless otherwise indicated.

\section{Recollections}\label{recollections} 
\subsection{Varieties for modules of a quantum elementary abelian group}
We recall some notation and results from \cite{PvW09}. 
Let $n$ be a positive integer, let $G$ be the abelian group $(\Z/\ell \Z)^{\times n}$ with generators $\{g_1, \ldots, g_n\}$, 
and let  $R = k[X_1, \ldots, X_n]/(X_1^\ell, \ldots, X_n^\ell)$ be a truncated polynomial ring. We let $G$ act on $R$ via the formula 
\[ g_i \cdot X_j = q^{\delta_{ij}}X_j\]
where $\delta_{ij}$ is the Kronecker delta.  
We set 
\[
A := R \rtimes G,
\]  the semi-direct product taken with respect to the action defined above.  
That is, $A$ is a free left $R$-module with basis $G$ and
multiplication determined by $(rg)(sh) = r (g\cdot s) gh$ for all
$r,s\in R$, $g,h\in G$. 
In addition,  $A$ has a Hopf algebra structure defined as follows:
\[ \Delta(X_i) = X_i \otimes 1 + g_i \otimes X_i, \ \ \ 
      \Delta(g_i) = g_i \otimes g_i, \]
\[\epsilon(X_i) =0, \ \ \ \epsilon (g_i) = 1, 
  \ \ \ S(X_i) = -g_i^{-1}X_i, \ \ \ S(g_i) = -g_i^{-1}. \] 
For $a \in A$, we use the Sweedler notation  $\Delta(a) = \sum a_1 \otimes a_2$. For two $A$-modules $M$, $N$, the tensor product $M\otimes N$ is equipped with an $A$-module structure via the formula
\[ a\cdot(m \otimes m') = \sum a_1m \otimes a_2m'\]
for all $a\in A$, $m\in M$, $m'\in N$. 
The $A$-module structure on a module dual  to $M$, denoted $M^\# = \Hom_k(M,k)$, is defined by 
\[(a \cdot f)(m) = f( S(a)m)	\]
for all $a\in A$, $f\in M^{\#}$, $m\in M$.
Equivalently, since $S^2$ is an inner automorphism (it is
conjugation by $g_1^{-1}\cdots g_n^{-1}$), we may define the action by $(a\cdot f)(m) =
f( S^{-1}(a)m)$.

Let $\{Y_1, \ldots, Y_n\}$ be non-commuting variables, and set 
\[R_q = \frac{k\langle Y_1, \ldots, Y_n \rangle}{(Y_jY_i -qY_iY_j, \, Y_i^\ell)}\]
where the commutator relations $Y_jY_i -qY_iY_j$ are taken for all $i,j$, $1 \leq i < j \leq n$. 
The group $G$ acts on $R_q$ via the same formula as before: $ g_i \cdot Y_j = q^{\delta_{ij}}Y_j$, and, hence, 
we can form a semi-direct product $R_q \rtimes G$.  There is an algebra isomorphism 
\[ R_q \rtimes G \stackrel{\sim}{\longrightarrow} R \rtimes G  = A \]
given by 
\[ Y_i \mapsto X_i g_1g_2\ldots g_{i-1}, \quad g_i \mapsto g_i. \]

For $\ul = [\lambda_1:\ldots:\lambda_n] \in \bP^{n-1}$  we define an embedding of algebras
\[ \tau_{\ul} : k[t]/(t^\ell) \to A  \quad \text{via} \quad t \mapsto \lambda_1Y_1 + \lambda_2Y_2 + \ldots + \lambda_nY_n.\]   
Let $\langle \tau_{\ul}(t)\rangle \subset A$ denote the subalgebra generated by $\tau_{\ul}(t)$.
We define the action of $G$ on $\bP^{n-1}$ by 
\[
g_i \cdot [\lambda_1:\ldots:\lambda_n] := [\lambda_1:\ldots:\lambda_{i-1}:q\lambda_i:\lambda_{i+1}:\ldots:\lambda_n].
\] 
For an $A$-module $M$, the restriction $M{\downarrow_{\langle \tau_{\ul}(t)\rangle}}$ of $M$
to $\langle\tau_{\underline{\lambda}}(t)\rangle$ is projective if and only if $M{\downarrow_{\langle \tau_{g\cdot \ul}(t)\rangle}}$ is projective for any $g \in G$ 
(see \cite[Lemma 2.5(ii)]{PvW09}). 
We  define the {\it rank variety} of $M$,  a closed subset  of $\bP^{n-1}$, denoted $V_A(M)$, as follows:
\begin{defn}{\cite[Defn.\ 3.2]{PvW09}}
\label{de:rank}
\[V_A(M):= \{\ul \in \bP^{n-1} \, | \, M{\downarrow_{\langle \tau_{\ul}(t)\rangle}} \text{ is not projective }\}/G.\]
\end{defn} 
\noindent 
Rank varieties have all the standard properties one expects (see \cite[Section 3]{PvW09}). 

For an $R$-module $M$, we set 
\[V_R(M) :=V_A(M{\uparrow^A}),\]
where $M{\uparrow^A} = \Ind_R^AM = A\otimes_R M $ is the (tensor) induced module.
In this paper we follow  standard conventions in  finite group theory, referring to the left adjoint functor to restriction as {\it induction}, and the right adjoint, $\Coind_R^A M = \Hom_R(A, M)$, as  {\it coinduction}.

Recall that the cohomology  of the quantum elementary abelian group $A$ is 
\begin{equation}\label{eq:coh}
\coh^*(A,k)=\Ext^*_A(k,k) \simeq k[y_1, \ldots, y_n],
\end{equation} 
where $\deg y_i =2$.  For an $A$-module $M$, denote by $I_A(M)$ the annihilator of $\Ext_A^*(M,M)$ under the (left) action of $\coh^*(A,k)$ by the cup product (equivalently, under the action given by tensoring an extension in $\Ext^*_A(k,k)$ with $M$ {\it on the right} followed by Yoneda composition). 
\begin{defn}
\label{de:cohom} 
The cohomological {\it  support variety} of $M$, $V_A^c(M)$,
is  the closed subset of $\bP^{n-1} \simeq \Proj \coh^*(A,k)$
defined by the homogeneous ideal $I_A(M)$, where $\Proj \coh^*(A,k)$ denotes the space
 of homogeneous prime ideals of $\coh^*(A,k)$ other than the ideal of all elements of positive degree. 
\end{defn}

We define a map $\widetilde \Psi:\bP^{n-1} \to \bP^{n-1}$
as 
$\widetilde \Psi([\lambda_1:\ldots:\lambda_n]) := [\lambda^\ell_1:\ldots:\lambda^\ell_n]$ and note that it factors through $\bP^{n-1}/G$. The resulting map is denoted by $\Psi$:
 
\[\xymatrix{&\bP^{n-1} \ar[dl]\ar[dr]^-{\widetilde\Psi}&  \\
\bP^{n-1}/G\ar[rr]^-\Psi && \bP^{n-1}} 
\]
Since we have $V_A = V_A(k) = \bP^{n-1}/G$ and  $V_A^c  = \Proj \coh^*(A,k) \simeq \bP^{n-1}$, we get the map 
\[\Psi: V_A \to V_A^c\]
induced by raising each coordinate of the equivalence class of $\ul \in V_A$ to the $\ell$th power.   The main theorem of \cite{PvW09} states that  $\Psi(V_A(M)) = V_A^c(M)$.  Since $\Psi$ is a homeomorphism, we identify the cohomological support variety $V_A^c(M)$ 
and rank  variety   $V_A(M)$ via the map $\Psi$.

We recall some standard properties of these varieties that we will need. The relative support variety $V_A^c(M,N)$ for $A$-modules $M$, $N$ is defined analogously to $V_A^c(M)$ by using the action of $\coh^*(A,k)$ on $\Ext_A^*(M,N)$ via $-\otimes M$ followed by Yoneda composition.

\begin{prop}\label{properties}
\cite[Prop.\ 4.3]{PvW09}
Let $M,N, M_1,M_2,M_3$ be $A$-modules.
\begin{itemize}
\item[(i)] $V_A(M\oplus N) = V_A(M) \cup V_A(N)$.
\item[(ii)] $V_A(\Omega(M)) = V_A(M)$, where $\Omega$ is the Heller shift operator.
\item[(iii)] If $0\rightarrow M_1\rightarrow M_2\rightarrow M_3\rightarrow 0$
is a short exact sequence of $A$-modules, then
$$V_A(M_i)\subset V_A(M_j)\cup V_A(M_l)$$ for any $\{i,j,l\}=\{1,2,3\}$. 
\item[(iv)] Let $S_1, \ldots, S_n$ be representatives of all isomorphism classes of simple $A$-modules.   Then 
$$V_A(M) = \bigcup_{i=1}^n V_A(S_i,M)=\bigcup_{i=1}^n V_A(M,S_i).$$ 
\end{itemize}
\end{prop} 

\begin{rem} 
We also claimed the inclusion $V_A(M\ot N) \subset V_A(M)\cap V_A(N)$ in \cite[Prop.~4.3]{PvW09}. 
This statement is true but the proof given was incomplete. In Section~\ref{product} we give a different proof of that inclusion as part of the proof of the ``Tensor Product Theorem~\ref{thm:tensor}'' which asserts the equality.  As demonstrated with counterexamples in \cite{BW}, one must be careful since the inclusion $V_A(M\ot N) \subset V_A(M)\cap V_A(N)$ does not hold for an arbitrary Hopf algebra. 

\end{rem}

\subsection{Thick subcategories} We recall a few basic definitions that we need and 
refer the reader to any of the multiple excellent sources on triangulated categories 
for the necessary background, such as \cite{HPS} or \cite{Ne}.  
 
 Let $T$ be a triangulated category. A {\it thick} subcategory $\cC \subset T$ is a full triangulated  subcategory that is closed under taking direct summands. If, in addition, $T$ is tensor triangulated (that is, a monoidal triangulated category), then a {\em thick tensor ideal} of $T$ is a thick subcategory $\cC$ satisfying the property that for any  $C \in \cC$, and any $B \in T$, 
we have both $B \otimes C\in \cC$ and $C \otimes B \in \cC$.

\begin{rem}
We could modify the above definition of thick tensor ideal to
define left and right tensor ideals.
It is the
notion of (two-sided) tensor ideal, as defined above, that 
naturally arises in the theory here due to the symmetry
inherent in the tensor product
property (Theorem~\ref{thm:tensor}) of varieties.
\end{rem}

For a triangulated category $T$ and an object $X \in T$, we denote by $\Thick_T(X)$ the thick subcategory generated by $X$. If $T$ is tensor triangulated, we denote by $\Thick^{\otimes}_T(X)$ the thick tensor ideal generated by $X$. 

Recall that the stable module category of $A$, $\stmod A$, is the category where objects are finitely generated 
$A$-modules  and morphisms are equivalence classes of morphisms in $\!\mod A$, where we say that $f: M\to N$ is equivalent to $g: M\to N$  if $f-g$ factors through a projective module. Since $A$ is a finite dimensional Hopf algebra, it is Frobenius (\cite{LS}), and, hence, projective modules coincide with injective modules. Therefore, the category $\stmod A$ is triangulated with exact triangles corresponding to short exact sequences in $\!\mod A$ and the shift functor given by the inverse Heller operator $ M \mapsto \Omega^{-1} M$  (see, for example, \cite{Ha}).

\section{Preliminaries and the variety of $M^\#$}
\label{se:dual}

In this section, we obtain some needed elementary results about
$A$-modules and their varieties.
Let $\Irr G$ denote the set of irreducible characters of $G$. 
For $\chi\in \Irr G$, write $S_{\chi}$ for the one-dimensional $A$-module
on which $x_i$ acts as 0 and $g_i$ acts as multiplication by $\chi(g_i)$. 
Note that $S_{\chi}^{\#}\simeq S_{\chi^{-1}}$. 

\begin{lem}\label{conjugate}
Let $M$ be an $A$-module and $\chi\in \Irr G$. Then  there is an
isomorphism of $A$-modules, 
$$S_{\chi}\ot M\ot S_{\chi}^{\#}
\simeq M.$$
\end{lem} 

\begin{proof}
The simple modules $S_{\chi}$ form a group, under tensor product, that is generated by 
all $S_{\chi_i}$, $1\leq i\leq n$, where $\chi_i(g_j) = q^{\delta_{ij}}$.
Thus it suffices to prove the statement for each character $\chi_i$.
Define a linear map 
$$\phi: S_{\chi_i}\ot M\ot S_{\chi_i}^{\#} \rightarrow  M$$
by $\phi( 1\ot m\ot 1) = g_i^{-1} m$ for all $m\in M$.
A straightforward calculation shows that $\phi$ is $A$-linear with inverse given by the formula 
$\phi^{-1}(m) = 1\ot g_i m\ot 1$ for all $m\in M$. Hence, it is an $A$-module isomorphism.
\end{proof}

As a consequence of the lemma, simple modules commute with all modules under
tensor product: $ \ S_{\chi}\ot M \simeq S_{\chi}\ot M\ot S_{\chi}^{\#}\ot S_{\chi}
\simeq M\ot S_{\chi}$ for all $\chi\in\Irr G$.

\begin{lem}\label{lem:simples}
Let $M$ be an  $A$-module. Then 
\begin{itemize}
\item[(i)]
$  M{\da_R\ua^A} \simeq\bigoplus_{\chi\in\Irr G} (M\ot S_{\chi}).$
\item[(ii)]  $(M{\downarrow_R\uparrow^A})^{\#} \simeq \bigoplus_{\chi\in\Irr G} (M^{\#} \ot S_{\chi})$.
\end{itemize}
\end{lem} 

\begin{proof}
We will give explicitly two $A$-module homomorphisms 
\[
   \phi \, : \,  M{\da_R\ua^A} \rightarrow \bigoplus_{\chi} (M\ot S_{\chi}), \quad \quad
   \psi \, : \, \bigoplus_{\chi} (M\ot S_{\chi}) \rightarrow M{\da_R\ua^A}.
\]
Let $g\in G$, $m\in M$, and
$$
  \phi(g\ot m) = (\chi(g)gm\ot 1) _{\chi\in\Irr G}.
$$
For each $\chi\in\Irr G$, let $m_{\chi}\in M$, and
$$
  \psi((m_{\chi}\ot 1)_{\chi\in\Irr G}) = \frac{1}{|G|} \sum_{\chi\in\Irr G}
     \sum_{g\in G} \chi(g^{-1}) g \ot g^{-1}m_{\chi}.
$$
It is straightforward to check that $\phi$, $\psi$ are mutually inverse $A$-module homomorphisms by 
applying the orthogonality relations for characters.

The second isomorphism follows immediately by dualizing the first (which reverses
the order of tensor products) and applying Lemma~\ref{conjugate}. 
\end{proof}

The following lemma only requires $A$ to be a finite dimensional Hopf algebra
for which the cohomological support varieties are defined and the two notions
of dual module (one using the antipode $S$ and the other using its inverse $S^{-1}$)
coincide. 
(If the two notions of dual are different, one must use $S$ for (i) below
and $S^{-1}$ for (ii) below.)
Since we work with left modules, we prefer the action given in (i),
however, sometimes we will need to compare with the action in (ii). 

\begin{lem}\label{lem:Ext}
There are isomorphisms of $\coh^*(A,k)$-modules for all $A$-modules $U$, $V$, $W$: 
\begin{itemize}
\item[(i)] $\Ext^*_A(U\ot V,W)\simeq \Ext^*_A(U, W\ot V^{\#})$, where the
action is $ - \ot U\ot V$ (respectively $ - \ot U$) followed by Yoneda composition.
\item[(ii)] $\Ext^*_A(U\ot V,W)\simeq \Ext^*_A(V, U^{\#}\ot W)$,
where the action is $ U\ot V\ot -$ (respectively $V\ot -$) followed
by Yoneda composition. 
\end{itemize}
\end{lem} 

\begin{proof}
Let $P_{\DOT}$ be a projective resolution of $k$ as an $A$-module, so that
$P_{\DOT}\ot U\ot V$ and $U\ot V\ot P_{\DOT}$ are projective resolutions of
$U\ot V$, and there are similar projective resolutions of $U$ and of $V$. 
One checks that the isomorphism of (i)  is induced by the chain level 
isomorphism
$$
  \phi: \Hom_A (P_i \ot U\ot V, W)\stackrel{\sim}{\longrightarrow}
   \Hom_A( P_i\ot U, W\ot V^{\#})
$$
given by $\phi(f) (x\ot u) = \sum_j f(x\ot u\ot v_j)\ot v_j^*$,
where $x\in P_i$, $u\in U$, and$\{v_j\}$, $\{v_j^*\}$ are dual bases for $V$.
Similarly one checks that the isomorphism of (ii) is induced by the chain
level isomorphism
$$
  \psi: \Hom_A(U\ot V\ot P_i ,W)\stackrel{\sim}{\longrightarrow}
   \Hom_A( V\ot P_i, U^{\#}\ot W)
$$
given by $\psi(f) (v\ot x) = \sum _j u_j^*\ot f(u_j\ot v\ot x)$,
where $v\in V$, $x\in P_i$, and $\{u_j\}$, $\{u_j^*\}$ are dual bases for $U$. 
\end{proof}

\begin{lem} \label{lem:res-ind}
Let $M$ be an $A$-module. Then
$$
   V_A^c(M{\da_R\ua^A}) = V_A^c(M).
$$
\end{lem}

\begin{proof}
By Lemma~\ref{lem:Ext}(i), there is an isomorphism of $\coh^*(A,k)$-modules, 
\begin{equation}
\label{eq:tensor}
\Ext_A^*(M\ot S_{\chi}, M\ot S_{\chi}) \simeq \Ext_A^*(M, M\ot S_{\chi} \ot S_{\chi}^\# ) 
  \simeq \Ext_A^*(M,M) , 
\end{equation} 
with the action of $\coh^*(A,k)$ given by tensoring  on the right
(by $M\ot S_{\chi}$, respectively, by $M$) 
followed by Yoneda composition. 
The isomorphism \eqref{eq:tensor} implies that
$V_A^c(M)=V_A^c(M\ot S_{\chi})$ for each $\chi\in\Irr G$. 
Applying Lemma~\ref{lem:simples}(i) and Proposition~\ref{properties}(i), we thus have
$V^c_A(M{\downarrow_R\uparrow^A})= \cup_{\chi\in\Irr G} V^c_A(M\ot S_{\chi})
=V^c_A(M)$. 
\end{proof}

Since $V_R(M{\da_R}) = V_A(M{\da_R \ua^A})$ by definition, Lemma~\ref{lem:res-ind} immediately implies the following. 
\begin{cor}
\label{cor:equal} Let $M$ be an $A$-module. Then 
\[V_R(M{\da_R}) = V_A(M).\]
\end{cor}

\begin{thm}\label{QEAG}
Let $M$ be an $A$-module. Then 
\[V_A^c(M)=V_A^c(M^{\#}).\]
\end{thm}

\begin{proof} 
By Lemma~\ref{lem:Ext}, there are isomorphisms of $\coh^*(A,k)$-modules: 
$\Ext_A^*(M^\#, k) \simeq \Ext_A^*(k, M^{\#\#}) \simeq \Ext_A^*(k,M)$ and, hence, 
$V_A^c(k,M) =V_A^c(M^{\#},k)$.

Now let $M$ be any  $A$-module, and 
$N = M{\downarrow_R\uparrow^A}$. By 
Proposition~\ref{properties}(iv) and Lemma~\ref{lem:simples}(i), 
$$
    V_A^c(N) =  \ \bigcup_{\chi\in\Irr G} V_A^c(S_{\chi},N) \ = \
     \bigcup_{\chi\in\Irr G} V_A^c(k, N\ot S_{\chi}^{\#}) \
   = V_A^c(k,N),
$$
since $N\ot S_{\chi}^{\#}\simeq N$ for all $\chi\in\Irr G$ 
(as $N\simeq \oplus_{\chi\in\Irr G}M\ot S_{\chi}$).
Similarly we find that $V_A^c(N) = V_A^c(N,k)$. 
It follows, from the previous observation, that
$V_A^c(N) = V_A^c(k,N) = V_A^c(N^{\#},k)$.
By Lemma~\ref{lem:simples}(ii), 
there is an isomorphism $N^{\#}\ot S_{\chi}^{\#} \simeq N^{\#}$ for each
$\chi\in\Irr G$, and so by a similar argument to the above,
we have $V_A^c(N^{\#})= V_A^c(N^{\#},k)$. Thus we have shown that
$V_A^c(N)=V_A^c(N^{\#})$.

By the above observations and 
Lemmas~\ref{lem:simples}(ii) and \ref{lem:res-ind},  we now have
\begin{eqnarray*}
  V_A^c(M) \ \ = \ \ V_A^c(M{\downarrow_R\uparrow^A})  &=&
   V_A^c((M{\downarrow_R\uparrow^A})^{\#})\\
    &=& V_A^c\left( \bigoplus_{\chi\in\Irr G} (M^{\#}\ot S_{\chi})\right) \\
    &=& V_A^c( M^{\#}{\downarrow_R\uparrow^A})  
    \ \ = \ \  V_A^c(M^{\#}).
\end{eqnarray*}
\end{proof}

\section{Tensor product property}\label{product}

In this section, we prove that for a quantum elementary abelian group $A$, 
the variety of a 
tensor product of  $A$-modules is the intersection of their varieties, as stated in Theorem~\ref{thm:tensor}. 
First, in Corollary~\ref{tp},
we fix the proof of Proposition~4.3(vi) from \cite{PvW09}, establishing the inclusion $V_A^c(M\ot N) \subset V_A^c(M)\cap V_A^c(N)$. Since the proof is entirely cohomological and applies more generally than for a quantum elementary abelian group, we stay with the notation $V^c_A$ for this part of the proof. 
The inclusion $V_A^c(M \ot N) \subset V_A^c(M)$ is a formal consequence of the definition of the cohomological support variety and the action of $\coh^*(A,k)$ on $\Ext_A^*(M,M)$. The inclusion for the second factor, $V_A^c(M\otimes N) \subset V_A^c(N)$, requires more careful consideration since the tensor product is not symmetric. Our approach is based on Theorem~\ref{QEAG} stating that $V_A^c(M) = V_A^c(M^\#)$ when $A$ is a quantum elementary abelian group. 

The proof of the opposite inclusion $V_A(M)\cap V_A(N) \subset V_A(M\ot N)$  uses both the cohomological and the non-cohomological description of the variety.   We observe that if $A$ is a quantum elementary abelian group, then the tensor product $A \otimes A$ is again a quantum elementary abelian group, and we check
in Lemma~\ref{inverse}  that the rank variety behaves well with respect to restriction along the map induced by the coproduct $A \to A\otimes A$. 
Using the naturality of the isomorphism between rank and support varieties with respect to the coproduct  as described in \eqref{eq:comm},  we deduce the ``restriction" property for the cohomological support varieties $V_A^c(M)$. The rest of the proof goes via standard arguments using the K\"unneth formula, and employs only the cohomological support.

\subsection{The inclusion $V_A^c(M\ot N) \subset V_A^c(M)\cap V_A^c(N)$} 
Similarly to Lemma~\ref{lem:Ext}, the next two results only require $A$ to be a finite dimensional Hopf algebra
for which the cohomological support varieties are defined (for example, it suffices to assume that $\coh^*(A,k)$ 
is finitely generated), so we state them in this greater generality. 
  
Recall that we define the action of $\coh^*(A,k)$ on $\Ext_A^*(M,M)$  via tensoring an extension with $M$ on the right followed by Yoneda composition, and then define the varieties $V_A^c(M)$ using that action.  In the following theorem we verify that defining the action by $M \otimes -$ followed by Yoneda composition will lead to the same result under the assumption that $V_A^c(M) = V_A^c(M^\#)$. 

\begin{thm}\label{same} Let $A$ be a finite dimensional Hopf algebra for which the cohomological support varieties 
are defined and the two notions of dual module (one using the antipode $S$ and the other using its inverse $S^{-1}$)
coincide. Let $M$ be an $A$-module for which $V_A^c(M) = V_A^c(M^{\#})$.  Then the closed subvariety of $\Proj \coh^*(A,k)$, defined by the annihilator ideal of the action of $\coh^*(A,k)$ on $\Ext_A^*(M,M)$ given by $M\ot -$ followed by Yoneda composition, coincides with $V_A^c(M)$.
\end{thm}

\begin{proof}
By Lemma \ref{lem:Ext}, 
the action of $\coh^*(A,k)$ on $\Ext^*_A(M,M)$ given by $M\ot -$ followed by
Yoneda composition corresponds to that on $\Ext^*_A(k, M^{\#}\ot M)$ given by
Yoneda composition, under the isomorphism $\Ext^*_A(M,M)\simeq \Ext^*_A(k, M^{\#}\ot M)$.
There is a further isomorphism $\Ext^*_A(k, M^{\#}\ot M) \simeq \Ext^*_A(k, M^{\#}\ot M^{\#\#})\simeq \Ext^*_A(M^{\#},M^{\#})$,
under which the action of $\coh^*(A,k)$ by Yoneda composition on the former corresponds
to that given on the latter by $ - \ot M^{\#}$ followed by Yoneda composition.
The variety $V^c_A(M^{\#})$ is defined by the annihilator of this action, and since
$V^c_A(M) = V^c_A(M^{\#})$, the variety $V^c_A(M)$ is defined by the annihilator of
the first action given above. 
\end{proof}

\begin{cor}\label{tp} Let $A$ be as in Theorem~\ref{same}.  
If $V_A^c(M)=V_A^c(M^{\#})$ for all  $A$-modules $M$,
then
$$
  V_A^c(M\ot N) \subset V_A^c(M)\cap V_A^c(N)
$$
for all  $A$-modules $M,N$.
\end{cor}

\begin{proof}
Since dualizing reverses the order of tensor product,
$$
  V_A^c(M\ot N) = V_A^c((M\ot N)^{\#}) = V_A^c(N^{\#}\ot M^{\#}).
$$
By the definition of support variety, $V_A^c(M\ot N)\subset V_A^c(M)$
and $V_A^c(N^{\#}\ot M^{\#})\subset V_A^c(N^{\#})$, and under
our hypothesis, $V_A^c(N^{\#})= V_A^c(N)$. 
\end{proof}

Now one inclusion in the tensor product theorem in the special case of a quantum elementary abelian group is an immediate consequence of Theorem~\ref{QEAG}  
and Corollary~\ref{tp}.
\begin{cor} Let $A$ be a quantum elementary abelian group, and let $M$, $N$ be two $A$-modules. Then $V_A^c(M\otimes N) \subset V_A^c(M) \cap V_A^c(N)$. 
\end{cor}

\subsection{The equality} 
For the rest of the paper, $A$ again denotes a quantum elementary abelian group.  
We choose the following ordered generating
set for $A\ot A$:
\begin{eqnarray*}
  &&X_1\ot 1 , \ 1\ot X_1, \ X_2\ot 1, \ 1\ot X_2, \ldots, \ X_n\ot 1, 
    \ 1\ot X_n,\\
  &&g_1\ot 1 , \ \ 1\ot g_1, \ \ g_2\ot 1, \ \ 1\ot g_2, \ldots, 
  \ \ g_n\ot 1, \ \   1\ot g_n.
\end{eqnarray*}
Under this choice, letting $\tau_{\ul}(t)\in A$ as
defined in Section \ref{recollections},
we find that in $A\ot A$,
$$
\tau_{(\lambda_1,\lambda_1,\ldots , \lambda_n,\lambda_n)}(t)
\hspace{5in}
$$

\vspace{-.2in}

\begin{eqnarray*}
&=& \lambda_1 X_1\!\ot\! 1 +\lambda_1g_1\!\ot\! X_1 + \lambda_2X_2g_1\!\ot\! g_1
  +\lambda_2g_1g_2\!\ot\! X_2g_1+\cdots \\
 &&\hspace{1in}+\lambda_n X_ng_1\cdots g_{n-1}\!\ot\!
  g_1\cdots g_{n-1}+\lambda_ng_1\cdots g_n\!\ot\! X_ng_1\cdots g_{n-1}\\
&=& \Delta(\tau_{(\lambda_1,\ldots,\lambda_n)}(t)).
\end{eqnarray*}
Accordingly we define a map on rank varieties,
 $\Delta^r: V_A\rightarrow V_{A\ot A}$, by
$$
  \Delta^r [\lambda_1:\lambda_2:\cdots : \lambda_n]
   = [\lambda_1:\lambda_1:\lambda_2:\lambda_2:\cdots :\lambda_n:\lambda_n].
$$

We denote the map on cohomology induced by the coproduct $\Delta: A \to A\ot A$ by 
\[\Delta^*: \coh^*(A,k)\otimes   \coh^*(A,k) \simeq \coh^*(A\ot A,k) \rightarrow \coh^*(A,k)\]
(where the first isomorphism is given by the K\"unneth theorem), 
and the corresponding map on support varieties by 
\[\Delta_*:V_A^c \to  V_{A \ot A}^c \simeq V_A^c \times V_A^c.\]
This map arises from the cup product on the graded commutative
algebra $\coh^*(A,k)$, and so is the diagonal map.

\begin{lem}\label{commutes}
The following diagram commutes:
\begin{equation}
\label{eq:comm}
\begin{xy}*!C\xybox{
\xymatrix{
  V_A \ar[r]^{\Delta^r} 
   \ar[d]^{\psi_A}  & 
  V_{A\ot A} \ar[d]^{\psi_{A\ot A}}\\
  V^c_A\ar[r]^{\Delta_*} & V^c_{A\ot A}
}}
\end{xy}
\end{equation}
\end{lem}

\begin{proof}
This is an immediate consequence of the definitions of the diagonal maps
$\Delta_*$ and  of $\Delta^r$.
(One must order the basis of $\coh^2(A\ot A,k)$
in accordance with the ordering of the generating set for $A\ot A$ used in
the definition of $\Delta^r$.)
\end{proof}

\begin{lem}\label{inverse}  Let $M$ be an
$A\ot A$-module, considered to be an $A$-module via $\Delta:
A\rightarrow A\ot A$. Then 
\begin{itemize}
\item[(i)] $(\Delta^r)^{-1} V_{A\ot A}(M) = V_A(M)$.
\item[(ii)] $(\Delta_*)^{-1} V^c_{A\ot A}(M) = V^c_{A}(M)$.
\end{itemize}
\end{lem}

\begin{proof}
(i) This follows from the definitions since
$$
  \tau_{(\lambda_1,\lambda_1,\ldots,\lambda_m,\lambda_m)}(t)
   = \Delta(\tau_{(\lambda_1,\ldots,\lambda_m)} (t)).
$$
Thus  $M{\downarrow^{A\ot A}_{k\langle \tau_{(\lambda_1,
\lambda_1,\ldots,\lambda_m,\lambda_m)}(t)\rangle}}$ is projective if, and only if,
$M{\downarrow^{A}_{k\langle \tau_{(\lambda_1,
\ldots,\lambda_m)}(t)\rangle}}$ is projective.

(ii) This follows from (i) and Lemma \ref{commutes}, since $\psi_A$ and
$\psi_{A\ot A}$ are homeomorphisms.
\end{proof}

The following result is a straightforward consequence of the K\"unneth
Theorem, valid for any finite dimensional Hopf algebra $A$ satisfying the
hypothesis of Theorem~\ref{same}, that is the
varieties of a module and of its dual are the same.

\begin{lem}\label{dp}
Let $M,N$ be  $A$-modules. Then
\[ \Spec \coh^*(A,k)/I_A(M) \times \Spec \coh^* (A,k)/I_A(N) \simeq 
\Spec \coh^*(A\ot A,k)/I_{A\ot A}(M\ot N) 
\]
\end{lem}
\begin{proof}
The following diagram
induces isomorphisms on varieties, as we explain below.

\vspace{0.1in}
\begin{changemargin}{-0.5cm}{-1cm} 
\noindent
$
\begin{xy}*!C\xybox{
\xymatrix@=0.9pc{
\Ext^*_A(k,k)\ot \Ext^*_A(k,k) \ar[rrrr]^{( M\ot  - )\ot ( - \ot N)} 
   \ar[dd]^{\simeq}  &&&&
  \Ext^*_A(M,M)\ot \Ext^*_A(N,N) \ar[dd]^{\simeq}\\
	&&&&\\
  \Ext^*_{A\ot A}(k,k)\ar@/^/[rr]^-{  M\ot k \ot - } 
    \ar@/_/[rr]_-{- \ot M\ot k} &&
   \Ext_{A\ot A}^*(M\ot k, M\ot k) \ar[rr]^-{ - \ot k\ot N}
   && \Ext^*_{A\ot A}(M\ot N,
    M\ot N)
}}
\end{xy}
$
\end{changemargin}
\vspace{0.1in}

\noindent 
The vertical maps are given by the K\"unneth Theorem.
The variety $V^c_A(M)\times V^c_A(N)$ is defined by the kernel of
the top horizontal map, since $M\ot -$ gives rise to the same
variety as $-\ot M$ by Theorem~\ref{same}. 
The variety $V^c_{A\ot A}(M\ot N)$ is defined by the kernel of
the bottom horizontal map. 
The middle horizontal map is part of a commuting diagram together with
the top horizontal map, by 
definition of the K\"unneth isomorphisms: 
These are simply different ways of viewing the same tensor product
of complexes, since $k$ is the trivial module.
Now $M\ot k\ot -$ gives rise to the same variety as $ - \ot M\ot k$,
again by Theorem~\ref{same}.
Therefore the bottom horizontal map
gives rise to the same variety as the 
middle horizontal map. 
\end{proof}

\begin{proof}[Proof of Theorem \ref{thm:tensor}]  It is more convenient to work with affine varieties  for this proof.   
Since $\coh^*(A,k)$ is graded, connected, and generated in degree 2, the ideal $I_A(M)$  for any  module $M$ is homogeneous, and the field $k$ is algebraically closed, we have 
\begin{equation}
\label{eq:Gm} 
V_A^c(M) = \left[\Spec\coh^*(A,k)/I_A(M)  - \{0\}\right]/  k^* 
\end{equation} 
where $\{0\}$  corresponds to the irrelevant ideal $\coh^{>0}(A,k)$.   Hence, it suffices to  show that $\Spec \coh^*(A,k)/I_A(M\otimes N) = \Spec \coh^*(A,k)/I_A(M) \cap \Spec \coh^*(A,k)/I_A(N)$. 

Recall that the map $\Delta^*:\coh^*(A,k) \ot \coh^*(A,k) \to \coh^*(A,k)$ induces the diagonal map on $\Spec$. 
Hence, \begin{eqnarray*}
\Delta_*^{-1}(\Spec \coh^*(A,k)/I_A(M) \times
   \Spec \coh^* (A,k)/I_A(N)) \\
    =  \Spec \coh^*(A,k)/I_A(M) \cap
   \Spec \coh^* (A,k)/I_A(N) .
   \end{eqnarray*}
By Lemma \ref{inverse}(ii) (which applies to affine varieties thanks to the observation \eqref{eq:Gm}), 
\[\label{eq2}\Spec \coh^*(A,k)/I_A(M\ot N)
  = (\Delta_*)^{-1} (\Spec \coh^*(A\ot A,k)/I_{A\ot A}(M\ot N)).
\]By Lemma \ref{dp}, and a reordering of the basis of $\coh^2(A\ot A,k)$, the latter equals 
\[\Delta_*^{-1}(\Spec \coh^*(A,k)/I_A(M) \times
   \Spec \coh^* (A,k)/I_A(N)).\]
Combining with the first equation, we conclude 
\[\Spec \coh^*(A,k)/I_A(M\otimes N) = \Spec \coh^*(A,k)/I_A(M) \cap \Spec \coh^*(A,k)/I_A(N).
\] 
Passing to projective varieties by \eqref{eq:Gm}, we get the desired property: 
\[V_A^c(M\otimes N) = V_A^c(M) \cap V^c_A(N).\]
\end{proof}

\begin{rem}
One consequence of the theorem involves 
modules of the form $V(\underline{\lambda})$, defined in \cite[\S2]{PvW09}
for each $\underline{\lambda} \in {\mathbb{P}}^{n-1}$:
It is shown there that the variety of $V({\underline{\lambda}})$ is
simply $\{{\underline{\lambda}}\}$. 
Let $M$ be an $A$-module. Then
$\ul \in V_A(M)$ if and only if $V(\ul)\ot M$ is not projective.
To see this, note that 
by Theorem \ref{thm:tensor}, since $V_A(V(\ul)) = \{\ul\}$, we have
$\ul\in V_A(M)$ if and only if 
$$
  V_A(V(\ul)\ot M) = V_A(V(\ul))\cap V_A(M) = \{\ul\} \neq \emptyset ,
$$
which happens if and only if $V(\ul)\ot M$ is not projective.
\end{rem}

\section{Classification of thick tensor ideals of $\stmod A$} 
\label{ideal}

The proof of the classification is based on the following general strategy (which can be traced to \cite{Ho}, see also \cite{Nee}) 
used recently in \cite{CI}. Let $A$ be any finite dimensional Hopf algebra such that finitely generated $A$-modules  have a support variety theory, denoted $V_A(M)$, satisfying the usual properties (such as in \cite[Prop.~4.3]{PvW09})  plus two additional hypotheses:
\begin{itemize}
\item[(i)] {\em (Realization)} For every closed subset $V$ in $V_A(k)$, 
there exists $M\in \!\! \mod A$ such that $V_A(M) = V$.
\item[(ii)] {\em (Thick subcategory lemma)}  Let $M$, $N$ be finitely generated $A$-modules. If $V_A(M) \subset V_A(N)$, then  $M \in \Thick^{\otimes}_A(N)$. 
\end{itemize}
In this case, the classification theorem holds.  We implement this strategy in the proof of Theorem~\ref{thm:class}. The realization property for our theory of support varieties holds by \cite[Cor.\ 4.5]{PvW09}, and the ``Thick subcategory lemma" is proved 
below as Proposition~\ref{thick_lemma}.

The following lemma is straightforward from the definitions (see also \cite[Lemma~2.2]{BIK}).
\sloppy{

}

\begin{lem}
\label{lem:functor}  Let $\cF : \cA \to \cB$ be an exact functor between small 
triangulated categories $\cA$ and $\cB$, and 
let $X$ be an object in $\cA$.  Then for any $Y \in \Thick_\cA(X)$, $\cF(Y)  \in  \Thick_{\cB}(\cF(X))$.
\end{lem}

\begin{lem}\label{lem:ind} Let $M \in \stmod A$.  Then $\Thick^{\otimes}_A(M) = \Thick^{\otimes}_A(M{\downarrow_R{\uparrow^A}})$.
\end{lem}

\begin{proof}
By Lemma \ref{lem:simples}, $M{\da_R\ua^A}\simeq \oplus_{\chi\in\Irr G}(M\ot S_{\chi})$.
Since $M$ is a direct summand of $M{\da_R\ua^A}$ (corresponding to the trivial 
character of $G$), $M \in\Thick^{\ot}_A(M{\da_R\ua^A})$.
Conversely, since $M\ot S_{\chi}\in\Thick_A^{\ot}(M)$ for each $\chi\in\Irr G$,
we find that $M{\da_R\ua^A}\in\Thick^{\ot}_A(M)$.
Hence, $\Thick^{\otimes}_A(M) = \Thick^{\otimes}_A(M{\downarrow_R{\uparrow^A}})$.
\end{proof}

The following proposition is a version of the ``thick subcategory lemma" for $V_R(M)$. It is proven in \cite{CI} but the definition of a variety used there is different from ours. 
So the proof is based on identifying the variety of \cite{CI} with $V_R(M)$ as defined in this paper.  
\begin{prop}\label{thick:R}
Let $M$, $N$ be  $R$-modules. If $V_R(M) \subset V_R(N)$, then  $$M \in \Thick_R(N).$$ 
\end{prop}
\begin{proof}
The varieties in \cite{CI} are the same as we have defined up to projectivization: 
Let $M$  be an $R$-module,
and identify it with a chain complex  concentrated
in degree 0. In \cite{CI},   $V_R(M)$ is defined to be the support, in
$\Spec k[\theta]$, of $\Ext^*_R(k,M)$, where $k[\theta]$ is the subalgebra of
$\Ext^*_R(k,k)$ generated in degree 2 by the usual generators. The
action of $\Ext^*_R(k,k)$ is the Yoneda action.
This is the same as the definition in
\cite{PvW09}, where  the support variety $V^c_R(M)$ (with
$R=\Lambda$ in that paper) 
is the closed subset of $V^c_R(k)$ determined 
by the annihilator of $\Ext^*_R(k,M)$ (see \cite[Definition 2, p.\ 587]{PvW09}). 
Since
the nilpotent generators (in degree 1) are contained in any
prime ideal, this is equivalent to  \cite[Definition~4.3]{CI}. See also the remark
on equivalent definitions, \cite[p.~587]{PvW09}. 

The result now follows from \cite[Theorem 5.6]{CI}.
\end{proof}

\begin{prop}  
\label{thick_lemma}  Let $M,N$ be  $A$-modules. 
If $V_A(M)\subset V_A(N)$, then
$$M\in \Thick^\otimes_A(N).$$
\end{prop}  
\begin{proof} 
By Corollary~\ref{cor:equal},  the inclusion $V_A(M)\subset V_A(N)$ implies 
 $V_R(M{\da_R})\subset V_R(N{\da_R})$.  Therefore, 
$M{\da_R} \, \in \Thick_R(N {\da_R})$ by Proposition~\ref{thick:R}. Since induction $(-){\ua^A}: \ \stmod R \to \stmod A$ is exact, we conclude that  
$M{\da_R\ua^A} \in \Thick_A(N {\da_R\ua^A})$ by Lemma~\ref{lem:functor}. By   
Lemma~\ref{lem:simples},  $M$ is a direct summand of $M{\downarrow_R\uparrow^A}$, and so 
$M \in \Thick_A(M{\da_R\ua^A})$.  Finally, $\Thick_A(N {\da_R\ua^A}) \subset \Thick_A^\otimes (N {\da_R\ua^A}) = \Thick^\otimes_A(N)$  by  Lemma~\ref{lem:ind}.  Putting these together, we conclude that $M \in \Thick^\otimes_A(N)$. 
\end{proof} 

Recall that a subset $V$ of a topological space $X$ is called {\it specialization closed} if for any $W \subset V$, the closure of $W$ belongs to $V$. 
Equivalently, a specialization closed subset is a union of closed subsets.

\begin{thm}\label{se:class}
There is an inclusion-preserving one-to-one correspondence

\vspace{0.1in} \noindent
$$
 \left\{\begin{array}{cc} 
     \mbox{thick tensor ideals } \cC \\
     \mbox{ in } \stmod A \end{array}\right\} 
   \longleftrightarrow
    \left\{ \begin{array}{cc}\mbox{specialization closed subsets } V \\
       \mbox{ of }{V_A (k)= \mathbb{P}}^{n-1}/G \end{array}\right\}
$$
given by maps $\phi,\psi$ 
where $$\phi(\cC) = \bigcup_{M\in \cC}V_A(M) \ \ \ 
   \mbox{ and } \ \ \ \psi(V)= \{M : V_A(M)\subset V\}.$$
\end{thm}
\begin{proof} 

First note that these maps have images as expected: If $\cC$ is a thick tensor
ideal, then  
by definition $\phi(\cC)$ is specialization closed, since 
$V_A(M)$ is closed for each $M$.
If $V$ is a specialization closed subset, then
by the properties of support varieties (Proposition~\ref{properties}), $\psi(V)$ is a thick tensor ideal. 

We wish to show that $\phi,\psi$ are inverse maps. We first show that $\phi\circ\psi(V) = V$. 
By definition, 
$$\phi\circ\psi(V) = \bigcup_{M \in \psi(V)}
   V_A(M),$$ 
where the union is taken over all $M$ for which $V_A(M)\subset V$. Hence,
$\phi\circ\psi(V) \subset V$. For the other
containment, note that any point $x$ in $V$ is contained in a closed set in 
$V$,
which is realized as the variety of some module $M$ (see \cite[Corollary~4.5]{PvW09}). Thus, $M$ is in
$\psi(V)$ by definition, and, therefore, $x \in \phi\circ\psi(V)$.

It remains to establish the equality $\psi\circ\phi(\cC) = \cC$.  The easy part of
this is $\cC\subset \psi\circ\phi(\cC)$, which follows from the definitions. So we only need to show the inclusion 
\[\psi \circ \phi(\cC) \subset \cC\]
for any thick tensor ideal $\cC$ in $\stmod A$. 
Let $N \in  \psi \circ \phi(\cC)$, that is, $V_A(N) \subset \bigcup_{M\in \cC}V_A(M)$.  
Since $\coh^*(A,k)$ is Noetherian, there exist finitely many $M_i \in \cC$ such that
\[
V_A(N) \subset \bigcup\limits_{i=1}^n V_A(M_i)
\]
(see \cite[Lemma 1.3]{Nee}).  Hence, $V_A(N)\subset V_A(M_1\oplus M_2\oplus\cdots
\oplus M_n)$. 
By  Proposition~\ref{thick_lemma}, 
$N \in \Thick^\otimes_A(M_1\oplus M_2 \oplus \ldots \oplus M_n)$. 
Since $M_1\oplus M_2\oplus\cdots\oplus M_n \in \cC$ and $\cC$ is a thick tensor ideal,
this implies $N\in \cC$.
\end{proof}

\begin{rem} As an immediate corollary to the classification, there is a one-to-one correspondence

\vspace{0.1in} \hspace{.5cm}
$\{ \ \mbox{thick tensor ideals in } \stmod A \ \}  \longleftrightarrow  \{ \ \mbox{thick subcategories in } \stmod R \ \}.$

\end{rem}

We briefly discuss the situation for the quantum complete intersection algebra $R_q$.  We note that $R_q$ is a Frobenius algebra, which follows from the fact that $R_q \rtimes G \simeq A$ (see also \cite[Cor.~5.8]{FMS}).  We consider the stable module category $\stmod R_q$. For an $R_q$-module $M$, and $g \in G$, denote by ${^g}M$ the $R_q$-module which is the ``twist" of $M$ by $g$: $\ {}^gM = M $ as a vector space, and the action is given by $ Y_i (m ) := (g^{-1}\cdot Y_i) m = 
\chi_i(g^{-1})Y_im $. 
(Recall that $\chi_i(g_j)= q^{\delta_{ij}}$.) 
Sending $M$ to ${}^gM$ induces an action of $G$ on $\stmod R_q$. We say that a subcategory $\cC$ of $\stmod R_q$ is $G$-{\it stable} if $M \in \cC$ implies ${}^gM \in \cC$ for all $g \in G$. 

For an $R_q$-module $M$, we define 
\[V_{R_q}(M) :=V_A(M{\uparrow^A}).
\]

\begin{cor}
There is an inclusion-preserving one-to-one correspondence 
$$\left\{ \begin{array}{cc}\mbox{ $G$-stable thick subcategories } \\
     \mbox{ in } \stmod R_q \end{array} \right\} \longleftrightarrow
   \left\{\begin{array}{cc} \mbox{ specialization closed subsets } \\
      \mbox{ of } \mathbb{P}^{n-1}/G \end{array}\right\} $$
induced by the maps $\phi, \psi$ defined as in Theorem~\ref{se:class} with $V_A$ replaced by $V_{R_q}$.  
\end{cor} 

\begin{proof}  The proof is very similar to the proof of Theorem~\ref{se:class}.  We discuss two points where differences occur. 

First, for any $R_q$-module $M$, we have $M{\uparrow^A} \simeq ({}^gM){\uparrow^A}$, and, therefore, $V_{R_q}(M) = V_{R_q}({}^gM)$  for any $g \in G$. The definition of $\psi$ now implies that for $V \subset \mathbb{P}^{n-1}/G$ a subset closed under specialization, $\psi(V)$ is a $G$-stable thick subcategory of $\stmod R_q$. 
 
As in the proof of Theorem~\ref{se:class}, it now suffices to show that for any $G$-stable thick subcategory $\cC \subset \stmod R_q$, we have $\psi(\phi(\cC)) \subset \cC$.  Let $M$ be an $R_q$-module such that $V_{R_q}(M) \subset \phi(\cC)$.  By the Noetherian argument as in the proof of Theorem~\ref{se:class} and the definition of the variety $V_{R_q}(-)$, there is a module $N \in \cC$ such that $V_A(M{\uparrow^A})\subset V_A(N{\uparrow^A})$, which
implies that $M{\uparrow^A}\in \Thick^{\ot}_A(N{\uparrow^A})$ by Proposition~\ref{thick_lemma}. Since the
$S_{\chi}$'s are all simple $A$-modules
up to isomorphism, we have $\Thick^{\ot}_A(N{\uparrow^A}) = \Thick_A(\oplus_{\chi} (N{\uparrow^A} \ot S_{\chi}))$.
This implies that $M{\uparrow^A\downarrow_{R_q}}\in\Thick_{R_q}(\oplus_{\chi}(N{\uparrow^A\ot
S_{\chi}){\downarrow_{R_q}}})$, however $(N{\uparrow^A\ot S_{\chi}){\downarrow_{R_q}}} \simeq \displaystyle{\oplus_{g \in G} } {}^gN$.
Since $N \in \cC$ and $\cC$ is $G$-stable, we conclude that $M{\uparrow^A\downarrow_{R_q}} \in \cC$. Therefore, $M \in \cC$ since $M$ is a direct summand of $M{\uparrow^A\downarrow_{R_q}}$. 
\end{proof} 

\begin{rem} Do there exist non-$G$-stable thick subcategories in $\stmod R_q$?
Or   do $N$, ${}^gN$ always generate the same thick subcategory of $\stmod R_q$, for any $N$?
In general, one does not expect this, however in this case, $N$ and $ {}^gN$ have
the same variety, and in case $q=1$, it is true.
This is related to a question that Benson and Green asked in \cite{BG04}:
The parameters were a bit different, but in our case their question
would translate as follows. 
Do there exist modules $M$, $N$, neither of which is free on
restriction to the subalgebra generated by $Y_1$,
and for which $\Ext^n_{R_q}(M,N)=0$ for all $n>0$? 
\end{rem}

\end{document}